\documentclass[12pt]{article}
\usepackage[a4paper,left=2cm,right=2cm,top=2cm,bottom=2cm]{geometry}

\usepackage[backend=bibtex8]{biblatex}
\usepackage{mathtools,amssymb}
\usepackage{amsthm}
\newtheorem{theorem}{Theorem}
\newtheorem*{theoremS}{Theorem \hypertarget{thm:s}{S}}
\newtheorem{proposition}{Proposition}
\newtheorem{lemma}{Lemma}
\newtheorem{corollary}{Corollary}

\theoremstyle{definition}
\newtheorem{definition}{Definition}

\newtheorem{remark}{Remark}
\newtheorem{example}{Example}

\usepackage{enumitem}

\usepackage[x11names]{xcolor}

\newcommand{\Z}{\mathbb{Z}}

\newcommand{\SL}[1]{\ensuremath{\mathrm{SL}_{#1}}}

\usepackage{mathrsfs}
\newcommand{\F}{\mathscr{F}}

\newcommand{\sA}{\left.\mathcal{A}_Q\right|_\sigma}
\newcommand{\triv}[1]{{\color{gray} #1}}
\newcommand{\fundom}[1]{\clap{\colorbox{gray!10}{\vphantom{1}#1}}}

\usepackage{tikz}
\usetikzlibrary{calc}
\usetikzlibrary{decorations.markings}
\usetikzlibrary{arrows.meta}
\usetikzlibrary{arrows,shapes,positioning}
\tikzstyle directed=[postaction={decorate,decoration={markings,mark=at position 0.7 with {\arrow{{latex}}}}}]
\tikzstyle directedbis=[postaction={decorate,decoration={markings,mark=at position 0.9 with {\arrow{{latex}}}}}]

\usepackage{doi}
\usepackage{hyperref}

\title{Beyond the Laurent phenomenon}
\author{Andrei Zabolotskii}
\date{}

\AtEndDocument{\par
	\medskip
	\begin{tabular}{@{}l@{}}
		{School of Mathematics and Statistics, The Open University,}\\
		{Milton Keynes, MK7 6AA, United Kingdom}\\
		\texttt{andrei.zabolotskii\raisebox{-0.4ex}{{\fontfamily{pag}\selectfont @}}open.ac.uk}
\end{tabular}}

\begin{document}

\maketitle

\begin{abstract}
In a cluster algebra, a subset of initial cluster variables can be specialised in such a way that all elements of the resulting algebra become polynomial in the remaining variables.
This allows obtaining any tame integral frieze pattern of type $A_n$ by specialisation of initial cluster variables, including specialisation to zero.
\end{abstract}

{\let\thefootnote\relax}

\section{Introduction}
Let $\mathcal{A}_Q\subset \mathbb{Q}(x_1,\ldots,x_n)$ be the cluster algebra associated with a quiver $Q$. The celebrated Laurent phenomenon \cite{CA1} is the surprising fact that all elements of $\mathcal{A}_Q$ are actually Laurent polynomials in $x_1,\ldots,x_n$ with integer coefficients, i.e., $\mathcal{A}_Q$ is a subalgebra of $\Z[x_1,\ldots,x_n,x_1^{-1},\ldots,x_n^{-1}]$.

In this note, we prove another fact which takes the Laurent phenomenon to the next level in the following sense: in $\mathcal{A}_Q$, when some of the initial cluster variables $x_1,\ldots,x_n$ are specialised in a certain controlled way, all elements of the resulting algebra become \emph{polynomial} (not Laurent polynomial) in certain initial cluster variables. This is made concrete in Theorem~\ref{thm:poly:main}. This is closely related to Berenstein, Fomin and Zelevinsky's results on upper bounds \cite{CA3}.

As an application, we generalise the known description of Conway--Coxeter friezes by cluster algebras of type $A_n$ and prove that all not-necessarily-positive tame integral frieze patterns arise from cluster algebras of type $A_n$ specialised this way (Theorem \ref{thm:poly:friezes}), using a combinatorial model for frieze patterns (Theorem \ref{thm:poly:model}). The combinatorial model is similar to but different from the models by Cuntz and Holm \cite{CuntzModel,CuntzHolmFriezes}.

We give some standard definitions from cluster algebra theory in section~\ref{sec:def} (which a reader familiar with this area can skip), prove Theorem~\ref{thm:poly:main} in section \ref{sec:poly}, and consider integral friezes patterns in section~\ref{sec:app}.

\section{Cluster algebras}
\label{sec:def}
Let $Q$ be a finite quiver without loops or oriented 2-cycles. We identify the set $Q_0$ of vertices of $Q$ with the numbers $\{1,\ldots,n\}$. For each vertex $j\in Q_0$, we introduce a variable $x_j$; together $(x_1,\ldots,x_n
)$ form the \emph{initial cluster} and the individual $x_j$s are the \emph{initial cluster variables}. The \emph{cluster algebra} $\mathcal{A}_Q$ associated with $Q$ is a $\Z$-algebra generated by certain elements of $\mathbb{Z}[x_1,\ldots,x_n,x_1^{-1},\ldots,x_n^{-1}]$ (details below), called \emph{cluster variables}, with an additional structure on the set of cluster variables: they are grouped into overlapping sets called \emph{clusters}. A cluster and the corresponding quiver together are called a \emph{seed} \cite{KellerClusterAlgQuivRep}.

Since their introduction in \cite{CA1}, cluster algebras have attracted immense interest because of numerous connections to other areas of mathematics, from quiver representations to integrable systems. While it is impossible to review all these connections in this short note, this section aims to show what is the Laurent phenomenon and why it is considered impressive.

For a chosen vertex $j\in Q_0$, a procedure called \emph{mutation} $\mu_j$ can be applied to a seed. Under $\mu_j$, the set of vertices of the quiver $Q_0$ is unchanged, while the arrows change as follows: all arrows incident to the vertex $j$ change direction; then all directed 2-paths with central vertex $j$ are completed to triangles; then any resulting oriented 2-cycles are removed. As for the action of $\mu_j$ on the cluster variables, the $j$-th cluster variable $x_j$ is replaced by
\begin{equation}
\label{eq:mut}
x_j' = \left.\left( \prod_{i\to j}x_i + \prod_{j\to i} x_i \right)\right/x_j,
\end{equation}
where the first (resp.\ second) product is over the arrows incoming to $j$ (resp.\ outgoing from $j$). All other cluster variables stay the same. $(x_1,\ldots,x_{j-1},x_j',x_{j+1},\ldots,x_n)$ is a new cluster. The new seed can be mutated again at different vertices repeatedly, generating more clusters. $\mathcal{A}_Q$ is the $\Z$-algebra generated by all cluster variables from all clusters. In fact, any seed can be chosen as initial: after applying all possible mutations, the algebra generated by all cluster variables will be the same, as well as the set of clusters; only the explicit expressions for the cluster variables in terms of the initial cluster variables will differ.

One important family of cluster algebras is the cluster algebras of type $A_n$, where $n$ is a positive integer. The cluster algebra of type $A_n$ is $\mathcal{A}_Q$ where $Q$ is a path with $n$ vertices and any orientation of arrows; all such paths are mutation equivalent, i.e., related by sequences of mutations.
\begin{example}[the cluster algebra of type $A_2$]
\label{ex:a2}
Let $Q$ contain just two vertices connected with a single arrow. Then the cluster variables arising from a single mutation $\mu_1$ or $\mu_2$ are  $x_1' = \tfrac{1+x_2}{x_1}$ and $x_2'=\tfrac{1+x_1}{x_2}$, respectively. Also, let $x_0=\tfrac{1+x_1+x_2}{x_1x_2}$. The clusters in $\mathcal{A}_Q$ turn out to be the initial cluster $(x_1,x_2)$, and also $(x_1',x_2)$, $(x_1,x_2')$, $(x_1', x_0)$, $(x_0,x_2')$ and the images of these five under swapping the order of the cluster variables. Thus, there are 5 cluster variables: $x_1$, $x_2$, $x_1'$, $x_2'$, and $x_0$.
\end{example}

Since mutation involves division by a cluster variable, it is natural to expect that after the second mutation, rational functions other than Laurent polynomials in the initial cluster variables $x_j$ arise among the cluster variables. Yet, this never happens. The \emph{Laurent phenomenon} is the property of cluster algebras, noted and proved in \cite{CA1}, which states that all elements of $\mathcal{A}_Q$ are actually Laurent polynomials in the initial cluster variables $x_1,\ldots,x_n$ (in fact, in the cluster variables from any chosen cluster) with integer coefficients; that is, $\mathcal{A}_Q$ is a~subalgebra of $\Z[x_1,\ldots,x_n,x_1^{-1},\ldots,x_n^{-1}]$.

\section{The polyamorous phenomenon}
\label{sec:poly}

Specialisations of initial cluster variables to~1s have been studied since \cite{CA4}, most recently in \cite[Section 4.2]{AdSGFriezes} in the context of friezes of cluster algebras.
However, the following definition also allows specialisation to $-1$s.

\begin{definition}
\label{def:spec}
A \emph{specialised cluster algebra} is the $\Z$-algebra arising from a cluster algebra associated with a quiver $Q$ 
by choosing an initial cluster and specialising some of the initial cluster variables  to $1$s or $-1$s.
A specialised cluster algebra is associated with a \emph{partially labelled quiver}, where the vertices of $Q$ corresponding to specialised cluster variables are labelled by the values assigned to these variables.
\end{definition}
The idea of specialisation of initial cluster variables to nonzero values other than 1 is of course not new; see, e.g., \cite[Section 5]{CuntzHolmFriezes}.

Typically, a specialised cluster algebra does not have the structure of a cluster algebra; it is just a ring generated by certain Laurent polynomials, unless all specialisations are to~1s (and sometimes to $-1$s, in certain degenerate cases) \cite[Lemma 6.6]{OnCatCA}.

A specialisation of a cluster algebra is determined by the chosen initial cluster and \emph{specialisation relations} of the form $x_{j_1}=a_1$, $x_{j_2}=a_2,\ldots$ (e.g.\ $x_2=1$, $x_5=-1$), which we will collectively refer to as $\sigma$, with the specialised cluster algebra denoted $\sA$ and the specialisation of a single element $x\in\mathcal{A}_Q$ denoted $x|_\sigma$. We identify specialisation relations $\sigma$ with the corresponding partial labelling of vertices of $Q$. The initial cluster variables not assigned a~value by specialisation will be referred to simply as \emph{variables} in $\sA$.

\begin{remark}[motivation for Definition~\ref{def:spec}]
Choosing an initial cluster and then specialising all initial cluster variables to~1s is a standard operation on cluster algebras giving rise to so-called unital frieze patterns, in particular to all Conway--Coxeter friezes; see Theorem~\ref{thm:poly:friezes} below. Specialising only to the units of $\Z$ guarantees that the elements of the specialised cluster algebra always remain Laurent polynomials \emph{with integer coefficients}; cf.\ \cite[Remark 2.2]{GunawanSchiffler}. However, generalising Definition~\ref{def:spec} to specialisation to any nonzero values would make sense too, and there exist specialisations that also ensure the integer coefficients, for example, specialising the cluster algebra of type $A_2$ with $x_1=2$, $x_2=3$ results simply in the ring $\Z$.
\end{remark}
\begin{remark}
The definitions and results of this section can be formulated in the more general setting of cluster algebras of geometric type associated with skew-symmetrisable matrices, which have been introduced in \cite{CA1}.
\end{remark}

\begin{definition}
In a specialised cluster algebra, a variable is \emph{polyamorous} if all elements of the algebra are polynomial in that variable. The vertex in the partially labelled quiver corresponding to a polyamorous variable is also called polyamorous. The specialised cluster algebra itself is called polyamorous when all variables in it are polyamorous.
\end{definition}
The name indicates that these variables ``love polynomials''.

The set of vertices adjacent to a given vertex in a quiver together with the vertex itself is called the \emph{neighbourhood} of that vertex.
\begin{definition}
\label{def:polycule}
In a partially labelled quiver associated with a specialised cluster algebra, the~neighbourhood of a vertex $v$ is a \emph{polycule} if $v$ is not labelled, each neighbour of $v$ is labelled, and the number of arrows connecting $v$ with the vertices labelled with $-1$ is odd.
\end{definition}

\begin{theorem}
\label{thm:poly:main}
In a partially labelled quiver associated with a specialised cluster algebra $\sA$, \\ a vertex is polyamorous if and only if its neighbourhood is a polycule.
\end{theorem}
\begin{proof}
Suppose $x_j$ is an initial cluster variable in $\mathcal{A}_Q$, $X$ is the collection of all other cluster variables in the initial cluster, $x_j'$ is the new cluster variable arising from the mutation of the initial cluster at $j$, and $y$ is some cluster variable in $\mathcal{A}_Q$. $y$ can be written as a Laurent polynomial $y(x_j,X)$. Suppose $\sigma$ is a labelling that makes the neighbourhood of $j$ a polycule. We need to show that $y|_\sigma$ is a polynomial in $x_j$.

Generally,
\begin{equation}
\label{eq:y}
y = \frac{a_nx_j^n + a_{n-1}x_j^{n-1}+\ldots+a_1x_j+a_0}{bx_j^m},
\end{equation}
where $a_n,\ldots,a_0$ are integer polynomials in $X$, $b$ is a monic monomial in $X$, and $m,n\geqslant0$. By the Laurent phenomenon, $y$ must remain a Laurent polynomial when expressed through the variables of the initial cluster mutated once at $j$, i.e.\ as a function of $x_j'$ and $X$. This implies that the coefficient of each power of $x_j'$ in $y$ expressed that way has to be a Laurent polynomial in $X$.

Per Eq.~\eqref{eq:mut}, $x_j'=P/x_j$ where $P$ is a polynomial in $X$ which is not a monomial (unless $j$ is an isolated vertex, but in that case its neighbourhood cannot be made a polycule) and does not have a nontrivial monomial factor (because $Q$ has no oriented 2-cycles). When $y$ is expressed through $x_j'$ and $X$, each term $a_kx_j^k$ in the numerator of the right-hand side of \eqref{eq:y} becomes $a_k(x_j')^{m-k}P^{k-m}/b$, so when $k<m$, necessarily $a_k = \tilde{a}_kP^{m-k}$ for some polynomial $\tilde{a}_k$ in $X$. However, when $\sigma$ makes the neighbourhood of $j$ a polycule, $P|_\sigma=0$. Therefore each term $a_kx_j^k$ in \eqref{eq:y} with $k<m$ turns into $0$ under specialisation $\sigma$, which makes $y|_\sigma$ a polynomial in $x_j$, as desired.

This establishes polynomiality of the specialisation of all cluster variables, and the polynomiality of all elements of $\sA$ follows immediately.

It remains to show that when $\sigma$ does not make the neighbourhood of $j$ a polycule, not all cluster variables are polynomial in $x_j$. Indeed, in that case $P|_\sigma\ne0$, so we have $x_j'|_\sigma=\left(P|_\sigma\right)/x_j$ non-polynomial in $x_j$.
\end{proof}
\begin{remark}
This argument can be viewed as an adaptation of the proof of \cite[Lemma 4.1]{CA3}, which can also be used to prove the theorem.
\end{remark}
\begin{corollary}
If $Q$ has at least one pair of vertices $j,k\in Q_0$ with an odd number of arrows between them then there exist specialisation relations $\sigma$ such that $\sA$ is polyamorous and the set of its variables $\{x_{j_1},\ldots,x_{j_\ell}\}$ is nonempty. Moreover, $\sA=\Z[x_{j_1},\ldots,x_{j_\ell}]$.
\end{corollary}
\begin{proof}
Label $k$ by $-1$ and the rest of the vertices in the neighbourhood of $j$ (except for $j$ itself) by $1$ or $-1$ such that the neighbourhood becomes a polycule (e.g.\ all by $1$). Label some other vertices (but not $j$) by $\pm1$ in an arbitrary way. Then also label all vertices whose neighbourhoods are not polycules by $\pm1$ in an arbitrary way (this doesn't affect polycules). All vertices that remained unlabelled, including $j$, are polyamorous due to Theorem~\ref{thm:poly:main}. Labels then determine the specialisation relations. Finally, since $\sA$ is polyamorous, it is a subalgebra of the polynomial algebra $\Z[x_{j_1},\ldots,x_{j_\ell}]$, and since the variables of $\sA$ belong to the set of its generators, $\sA$ contains the entire polynomial algebra.
\end{proof}
\begin{corollary}
Suppose $Q$ has no multiple arrows and no isolated vertices (e.g.\ any Dynkin quiver of type $A,D,E$ except for $A_1$) and $A$ is a subset of $Q_0$ with pairwise disjoint neighbourhoods. Then there exists a labelling $\sigma$ of the vertices $Q_0\setminus A$ which makes all vertices from $A$ polyamorous.
\end{corollary}
\begin{proof}
The neighbourhood of each vertex from $A$ can be labelled independently since they are all disjoint. The conditions on $Q$ ensure that each neighbourhood can be made a polycule, which allows us to apply the theorem.
\end{proof}

\begin{example}
\label{ex:spec}
In the cluster algebra of type $A_2$ of Example~\ref{ex:a2}, generated by the cluster variables $x_1$, $x_2$, $\frac{x_1+1}{x_2}$, $\frac{x_2+1}{x_1}$, $\frac{x_1+x_2+1}{x_1x_2}$, specialise $x_2=-1$. The neighbourhood of vertex~1 becomes a polycule, and the cluster variables are specialised resp.\ to $x_1,\ -1,\ -x_1-1,\ 0,\ -1$; hence the specialised cluster algebra becomes $\Z[x_1]$.
\end{example}

For a given quiver, it is possible to characterise all specialisations that give rise to polyamorous specialised cluster algebras in terms of a system of linear equations over $\Z/2\Z$.
\begin{proposition}
For a quiver $Q$ with $n$ vertices and specialisation relations $\sigma$ that specialise $m$ out of $n$ variables, let $B$ be the antisymmetric adjacency matrix of $Q$ and let $M$ be the $(n-m)\times m$ minor of $B$ in which columns correspond to the variables specialised by $\sigma$ and rows correspond to the remaining variables. The specialised cluster algebra $\sA$ is polyamorous if and only if the vertices corresponding to the variables of $\sA$ form an independent set and $\bar{M}\mathbf{s}=\mathbf{u}$, where $\bar{M}$ is $M$ reduced modulo $2$ entrywise, the binary vector $\mathbf{s}$ of size $m$ is defined by $x_j=(-1)^{s_j}$, where the index $j$ runs over the specialised variables, and $\mathbf{u}$ is the binary vector of size $n-m$ of all $1$s.
\end{proposition}
\begin{proof}
By Theorem \ref{thm:poly:main}, $\sA$ is polyamorous if and only if each non-specialised vertex has only specialised neighbours -- which is equivalent to the condition that non-specialised vertices form an independent set -- and the number of arrows connecting a non-specialised vertex to vertices specialised to $-1$ is odd, which is equivalent to the condition that the corresponding component of the binary vector $\bar{M}\mathbf{s}$ is equal to~$1$.
\end{proof}

\section{Application: integral friezes}
\label{sec:app}

\subsection{Conway--Coxeter friezes}
The cluster algebra of type $A_n$ can be described using triangulations of an $N$-gon (hereafter $N=n+3$) instead of quivers. For a given triangulation, put a vertex of the quiver on each of the $n$ diagonals of the triangulated $N$-gon and connect two vertices with an arrow whenever the corresponding diagonals form two sides of a triangle and their shared vertex of the $N$-gon is to the left of the arrow.

When a triangulation has no internal triangles, the corresponding quiver is an orientation of a path, as required for the cluster algebra of type $A_n$. The mutation of the quiver at a vertex corresponds to the flip of the corresponding diagonal, that is, to replacing it with the other possible diagonal in the enclosing quadrilateral. Clusters correspond to different triangulations, while cluster variables correspond to different diagonals of the polygon.

Triangulations of an $N$-gon are also in bijection with objects called Conway--Coxeter frieze patterns, or simply friezes.

A \emph{Conway--Coxeter frieze} \cite{ConwayCoxeter} is an array consisting of bi-infinite rows arranged such that consecutive rows are shifted relatively to each other (as shown in the examples below), the first and the last rows consist of $0$s, the second and the penultimate rows consist of $1$s, and rows in between them (to be referred to as \emph{nontrivial rows}) consist of positive integers such that any $2\times2$ diamond $\begin{smallmatrix}&b\\a&&d\\&c\end{smallmatrix}$ has determinant one: $ad-bc=1$ (the \emph{diamond rule}).
\begin{example}
\label{ex:ccf}
A frieze with two nontrivial rows (which happens to be unique for this number of rows, up to horizontal shifts) and the corresponding triangulation of a pentagon, with the quiver, that comprises a single arrow, superimposed; the triangulation can also be viewed as the interior of a cycle in the Farey graph $\F$ (to be defined below), ignoring the geometry of $\F$.

\begin{center}
\(
\begin{matrix}
&\triv0 && \triv0 && \triv0 && \triv0 \\
&& \triv1 && \triv1 && \triv1 && \triv1 \\
\cdots&3 && \fundom1 && \fundom2 && \fundom2 && \cdots \\
&& 2 && \fundom1 && \fundom3      && 1 \\
&\triv1 && \triv1 && \triv1 && \triv1 \\
&& \triv0 && \triv0 && \triv0 && \triv0 \\ 
\end{matrix}
\)
\qquad
\begin{tikzpicture}[scale=1.5, baseline=(current bounding box.center), >={Stealth}]
  \coordinate (A) at (90:1);
  \coordinate (B) at (18:1);
  \coordinate (C) at (-54:1);
  \coordinate (D) at (-126:1);
  \coordinate (E) at (162:1);

  \draw[very thick] (A) -- (B) -- (C) -- (D) -- (E) -- cycle;

  \draw[thick] (A) -- (C);
  \draw[thick] (A) -- (D);

  \coordinate (x1) at ($(A)!0.5!(C)$);
  \coordinate (x2) at ($(A)!0.5!(D)$);

  \fill[gray] (x1) circle (0.05);
  \fill[gray] (x2) circle (0.05);
  \draw[->, gray, very thick] (x2) -- (x1);
  \node[right=0.1] at (x1) {$x_1$};
  \node[left]  at (x2) {$x_2$};
\end{tikzpicture}
\end{center}
\end{example}
Any Conway--Coxeter frieze is \emph{tame}, which means that any $3\times3$ diamond in it has determinant zero.

Also, any Conway--Coxeter frieze with $n$ nontrivial rows possesses a~glide reflection symmetry with the horizontal glide by $N/2$ positions. This implies that a fundamental domain for that symmetry includes $\left(\begin{smallmatrix}n+2\\2\end{smallmatrix}\right)-1$ elements from nontrivial rows, such as the elements highlighted in Example~\ref{ex:ccf}.

The elements of a frieze with $n$ nontrivial rows are indexed by two indices $i,j\in\Z$ with $0\leqslant i-j \leqslant N$ such that the top row of $0$s contains elements with $i=j$, the index $i$ increases along the diagonal in the bottom-right direction as $j$ stays constant, and the index $j$ decreases along the diagonal in the bottom-left direction as $i$ stays constant. If the vertices of the triangulated $N$-gon associated with the frieze are enumerated in the clockwise order, the $(i,j)$th element of the frieze corresponds to the diagonal connecting the vertices $i$ and $j$ (mod $N$). The glide reflection symmetry ensures that the different elements associated with the same diagonal are all the same. One way to specify the aforementioned bijection from triangulations to friezes is as follows: take the seed of the cluster algebra of type $A_n$ that corresponds to the triangulation; then the $(i,j)$th element of the frieze is equal to the cluster variable associated with the corresponding diagonal when each initial cluster variable is specialised to~1 \cite{CC06}. Note that the Laurent phenomenon ensures that the resulting numbers are all integers.

\begin{example}
\label{ex:genfr}
The cluster variables from a cluster algebra of type $A_2$, listed in Example~\ref{ex:a2}, can be arranged as
\[
\begin{matrix}
x_1 && \frac{1+x_2}{x_1} && \frac{1+x_1}{x_2} \\
& x_2 && \frac{1+x_1+x_2}{x_1x_2}
\end{matrix}
\]
Upon specialisation $x_1=1$, $x_2=1$, this becomes the fundamental domain highlighted in Example~\ref{ex:ccf}.
\end{example}
By representing the mutation relation between cluster variables \eqref{eq:mut} in division-free form, one obtains the \emph{Ptolemy relations} between frieze entries: for four cyclically ordered vertices $\alpha\leqslant\beta\leqslant\gamma\leqslant\delta\leqslant\alpha+N$, the frieze entries $f_{i,j}$ satisfy
\[
f_{\gamma,\alpha}f_{\delta,\beta} = f_{\beta,\alpha}f_{\delta,\gamma}+f_{\gamma,\beta}f_{\delta,\alpha}.
\]
Both the diamond rule and the tameness condition for a frieze follow from the Ptolemy relations; the specialisation to~1s plays no role in it.

Another view on friezes comes from hyperbolic geometry. The Farey graph $\F$ is a graph embedded into the closure of the hyperbolic plane $\bar{\mathbb H}$. The vertices of $\F$ are the rational numbers $\tfrac a b$, identified with the points on the real axis which is the boundary of $\bar{\mathbb H}$ in the upper half-plane model, together with $\tfrac10$; the edges connect vertices $\tfrac a b$ and $\tfrac c d$ whenever $|ad-bc|=1$ and are realised as geodesic lines in $\bar{\mathbb H}$. $\F$ forms a tessellation of $\bar{\mathbb H}$ by ideal triangles. Thus, the interior of any length-$N$ cycle in $\F$ is an ideal triangulated $N$-gon. The triangulated $N$-gons are in bijection with the length-$N$ cycles in $\F$ taken up to the orientation-preserving automorphisms of $\F$, and therefore so are the Conway--Coxeter friezes with $n$ nontrivial rows \cite{MoOvTa2015}, and so are the clusters of the cluster algebra of type~$A_n$. These automorphisms are all induced by the action of the modular group; we will use the action of $\SL2(\Z)$ instead because it also acts on standard lifts (by left multiplication).

In a circuit (closed walk) in $\F$ with consecutive vertices numbered by $0,\dots,$ \mbox{$N-1$}, suppose the edge from the $i$th vertex to the $(i+1)$th vertex carries a label $\lambda_i = 1$ or $-1$ for each $i$; then the cyclic sequence of vertices $\big(\tfrac{a_0}{b_0},\dots,\tfrac{a_{N-1}}{b_{N-1}},\allowbreak\tfrac{a_N}{b_N}=\tfrac{a_0}{b_0}\big)$ can be lifted to a sequence of pairs of integers $\begin{psmallmatrix}a_i\\b_i\end{psmallmatrix}$ such that $\det\begin{psmallmatrix}a_i&a_{i+1}\\b_i&b_{i+1}\end{psmallmatrix}=\lambda_i$ for $i=0,\dots,N-1$; there are then two possibilities: $\begin{psmallmatrix}a_N\\b_N\end{psmallmatrix} = \begin{psmallmatrix}a_0\\b_0\end{psmallmatrix}$ or $\begin{psmallmatrix}a_N\\b_N\end{psmallmatrix} = \begin{psmallmatrix}-a_0\\-b_0\end{psmallmatrix}$, depending only on the labels and the $\SL2(\Z)$ class of the circuit. We will say that this sequence is a \emph{standard lift} of the labelled circuit if $\begin{psmallmatrix}a_N\\b_N\end{psmallmatrix} = \begin{psmallmatrix}-a_0\\-b_0\end{psmallmatrix}$. A clockwise simple circuit labelled with 1s always has a standard lift. (In terms of \cite{OUmodular}, such standard lift is a walk in $\mathscr{E}_{\Z}$.)
It may appear that the position of the zeroth vertex in the circuit is significant in this definition, however, if we allow any integer as the vertex numbers $i$ (identifying those that are congruent mod $N$), while the sequence of fractions $\tfrac{a_i}{b_i}$ is $N$-periodic, the standard lift can be made $N$-antiperiodic, so the existence of a standard lift does not depend on the choice of the zeroth vertex in the circuit.

The bijection from $\SL2(\Z)$-equivalence classes of cycles, or clockwise simple circuits, to the set of Conway--Coxeter friezes is given by the map $\Phi$ which sets the $(i,j)$th frieze entry to be equal to $\det\begin{psmallmatrix}a_j&a_i\\b_j&b_i\end{psmallmatrix}$ for $0\leqslant j\leqslant i<N$, where $\begin{psmallmatrix}a_i\\b_i\end{psmallmatrix}_{i=0,\dots,N-1}$ is the standard lift of any circuit from the equivalence class with all edge labels set to~1.

See \cite{ourApp} for an interactive visual demonstration of the relations between friezes, cluster algebras of type $A_n$, and other objects.

\subsection{Tame integral friezes}
We will now consider \emph{tame frieze patterns over $\Z$}, which generalise Conway--Coxeter friezes by allowing zero and negative entries in the nontrivial rows; they are also known as [not-necessarily-positive] \emph{integral friezes}. The polyamorous phenomenon enables us to establish a connection between tame frieze patterns over $\Z$ and cluster algebras of type $A_n$ generalising the analogous result for Conway--Coxeter friezes that was outlined above.

Tame integral friezes were studied by Short \cite{sl2short} and Cuntz and Holm~\cite{CuntzHolmFriezes}. The following is essentially \cite[Theorem 1.6]{sl2short} (or a special case of \cite[Theorem 1.3]{OUmodular}).
\begin{theoremS}
$\Phi$, defined as above, is a bijection from the set of $\SL2(\Z)$-equi\-valence classes of length-$N$ (not necessarily simple) circuits in $\F$ that have a standard lift to the set of tame integral (not necessarily positive) friezes with $n$ nontrivial rows.
\end{theoremS}
\begin{remark}
Short actually considers all circuits in $\F$, but circuits that do not have a standard lift produce what is called tame semiregular friezes in \cite{OUmodular} (and referred to simply as friezes in \cite{sl2short}).
\end{remark}

Cuntz and Holm's \cite[Theorem 5.4]{CuntzHolmFriezes} establishes that for almost any tame frieze, one can choose an initial cluster whose initial cluster variables are all nonzero in that frieze, and therefore the entire frieze arises from cluster variables through specialisation, but it does not tell us what that specialisation should be for the frieze to be integral (however, see also Remark \ref{remark:poly}\ref{remark:poly:CuntzHolm} below). So, the following theorem is a complementary result: while \cite[Theorem 5.4]{CuntzHolmFriezes} says that specialisation to zero \emph{can} be avoided, we show that it \emph{need not} be avoided completely and makes sure that the resulting frieze is integral.

\begin{theorem}
\label{thm:poly:friezes}
Consider the cluster algebra of type $A_n$, for some $n$. Specialise it so that the resulting specialised cluster algebra is polyamorous. Specialise further the remaining variables to any values from $\Z$. The resulting values of cluster variables form a tame frieze pattern over~$\Z$. Moreover, every tame frieze pattern over $\Z$ arises this way.
\end{theorem}
This result is also hinted at in \cite[Remark 2.12(b)]{CuntzModel}, but it needs the definitions and results from the preceding section to be made precise.

The proof will require some preparation.
\begin{lemma}
\label{lemma:cycle34}
Any circuit in $\F$ of length greater than 4 has a pair of non-consecutive adjacent vertices. A circuit of length 4 which does not have a pair of non-consecutive adjacent vertices has a pair of coinciding vertices (such a circuit will be referred to as a \emph{bispur}).
\end{lemma}
\begin{proof}
In a circuit of length greater than 4, consider consecutive vertices $(x,y,z,u,v)$. If $y$ and $u$ are adjacent then they are the desired pair; if $y=u$ then the desired pair is $x$ and $u$ or $y$ and $v$; otherwise, there is a ``separating'' vertex $w$ of $\F$ adjacent to $z$ such that $y$ and $u$ are on the different sides of the edge connecting $z$ and $w$. When going back from $u$ to $y$, the circuit has to pass either through $w$, in which case the desired pair is $z$ and $w$, or through $z$, in which case it is adjacent to $y$ and to $u$, and forms the desired non-consecutive pair with at least one of them.

For a length-4 circuit $(y,z,u,v)$, either $y=u$, or (since we assumed they cannot be adjacent) there is a ``separating'' vertex $w$ adjacent to $z$; then $v\ne w$, which implies $v=z$, as desired.
\end{proof}

We will now define a combinatorial model for integral friezes that generalises Conway and Coxeter's correspondence between triangulated polygons and their positive integral friezes. We will label the diagonals by numbers and then specialise the corresponding initial cluster variables to these labels.

Note that for any diagonal of a triangulated polygon, the quadrilateral enclosing it corresponds precisely to the neighbourhood of the vertex of the quiver associated with the triangulated polygon as explained in the previous section. To be more precise, sides of the polygon do not correspond to vertices, but they can as well be considered ``frozen'' vertices whose initial cluster variables are specialised to~1 and mutations at these vertices are not allowed -- this does not change the mutation relations for the ordinary vertices. Accordingly, we will call a diagonal together with its enclosing quadrilateral a \emph{polycule} whenever the sides of the quadrilateral are labelled by 1s and $-1$s and the number of $-1$s is odd; cf.\ Definition~\ref{def:polycule}.
\begin{definition}
\label{def:labtr}
A \emph{labelled polygon} is a polygon with some (possibly none) diagonals drawn, each of these diagonals is labelled by an integer, and each side is labelled by~1 or $-1$, so that each diagonal whose label is different from $\pm1$ (to be referred to as a \emph{non-unit diagonal}) is enclosed in a polycule. A~labelled polygon whose diagonals triangulate it and each side is labelled by~1 is a \emph{labelled triangulated polygon} (or $N$-gon).
\end{definition}
\begin{definition}
Given a labelled polygon, a mapping from its vertices to the vertices of $\F$ is \emph{compatible with labelling} if its boundary is mapped to a circuit that has a standard lift $\begin{psmallmatrix}a_i\\b_i\end{psmallmatrix}_{i=0,\dots,N-1}$, inheriting the labels of the edges from the labelled polygon, and the label of any diagonal connecting the vertices $i$ and $j$ with $0\leqslant j< i<N$ is equal to $\det\begin{psmallmatrix}a_j&a_i\\b_j&b_i\end{psmallmatrix}$. A mapping is \emph{poly-compatible with labelling} if it is compatible and for any non-unit diagonal in it, the two vertices of the enclosing polycule not incident to the non-unit diagonal are mapped to the same vertex of $\F$.
\end{definition}
A compatible mapping of a labelled polygon with all labels $\pm1$ is a graph homomorphism, but generally not an embedding.
Compatibility with labelling is preserved by $\SL2(\Z)$ transformations acting on $\F$. Like a standard lift, a compatible mapping does not in fact single out the zeroth vertex. Note also that if a map is (poly-)compatible with labelling of a labelled polygon, it remains (poly-)compatible with labelling when restricted to the vertices of any subset of edges that forms a polygon itself (a \emph{subpolygon}).

\begin{theorem}
\label{thm:poly:model}
\begin{enumerate}[label=(\roman*)]
\item \label{it:compat}
For any labelled triangulated $N$-gon $T$, there is a map from its set of vertices to the vertices of $\F$ poly-compatible with its labelling. 
\item \label{it:frieze}
The boundary of $T$ gets mapped to a circuit whose image under $\Phi$ is a tame integral frieze; if vertices $i$ and $j$ are connected by a diagonal then its label is equal to the $(i,j)$th entry of the frieze.
\item \label{it:surj}
All tame integral friezes arise from labelled triangulated polygons this~way.
\end{enumerate}
\end{theorem}
\begin{proof}
\begin{enumerate}[label=(\roman*)]
\item
Triangles with edges labelled by $\pm1$ and polycules enclosing a non-unit diagonal will be called elements. $T$ consists of elements, whose adjacency graph (where two elements are adjacent if they share an edge) is a plane tree. It is easy to verify that for each element there exists a mapping to $\F$ poly-compatible with its labelling, unique up to $\SL2(\Z)$-transformations. For example, a triangle with each side labelled by a~1 is mapped to a clockwise triangle, a triangle labelled by $-1$s is mapped to a counterclockwise triangle, and a polycule is mapped to a bispur. (Note that an edge-labelled bispur has a standard lift if and only if its labelling comes from a polycule.) For any two adjacent subpolygons of $T$ with unique up to $\SL2(\Z)$ poly-compatible mapping, there is a unique poly-compatible mapping for the subpolygon which is their union, because the shared edge completely determines their relative position. Thus one can build a poly-compatible mapping for the entire $T$ by iteratively building it up from elements, and the resulting construction is unique up to $\SL2(\Z)$-transformations.
\item
Follows directly from the definitions and Theorem~\hyperlink{thm:s}{S}.
\item
By Theorem \hyperlink{thm:s}{S}, for every integral frieze, there is a circuit in $\F$ in its preimage under $\Phi$. This circuit corresponds to the sides of $T$. By Lemma~\ref{lemma:cycle34}, if $N>4$, we find a pair of adjacent non-consecutive vertices in the circuit; we connect the corresponding vertices of $T$ by a diagonal. Hereafter we label newly added diagonals by the corresponding entries of the frieze, which in this case are either 1 or $-1$. This breaks the big circuit into smaller circuits, and we keep doing this recursively, adding diagonals to $T$, until only triangles and length-4 circuits are left and no more diagonals can be added. Lemma~\ref{lemma:cycle34} ensures that each of these length-4 circuits is a bispur; so we take a pair of opposite vertices such that the other pair coincides and connect it with a (non-unit) diagonal. The Ptolemy relation then ensures that the neighbourhood of this diagonal is a polycule in $T$.
Since the poly-compatible mapping from $T$ is unique up to $\SL2(\Z)$, it produces the same circuit, and therefore the same frieze from which we started, as desired.
\qedhere
\end{enumerate}
\end{proof}

\begin{example}
Left: a labelled triangulated polygon $T$. The non-unit diagonals are shown dashed. $T$ can be partitioned into two polycules and two triangles. Centre: mapping of $T$ to the Farey graph $\F$ poly-compatible with the labelling. Only some vertices and edges of~$\F$ are shown, disregarding its geometry. The boundary of $T$ becomes a circuit $\tfrac10,\tfrac01,\tfrac11,\dots$ which has a standard lift $\begin{psmallmatrix}1\\0\end{psmallmatrix},  \begin{psmallmatrix}0\\1\end{psmallmatrix}, \begin{psmallmatrix}-1\\-1\end{psmallmatrix}, \begin{psmallmatrix}-2\\-3\end{psmallmatrix}, \begin{psmallmatrix}-3\\-5\end{psmallmatrix}, \begin{psmallmatrix}2\\3\end{psmallmatrix}, \begin{psmallmatrix}-3\\-4\end{psmallmatrix},$ $\begin{psmallmatrix}1\\1\end{psmallmatrix},$ $\begin{psmallmatrix}-1\\0\end{psmallmatrix}$. When two vertices of $T$ are mapped to the same vertex of $\F$, they are shown slightly apart. Right: the fundamental domain of the corresponding frieze. For example, the bottommost diagonal in $T$ labelled $-1$, when mapped to $\F$ and lifted to the standard lift, connects $\begin{psmallmatrix}-1\\-1\end{psmallmatrix}$ to $\begin{psmallmatrix}2\\3\end{psmallmatrix}$, and $\det\begin{psmallmatrix}-1&2\\-1&3\end{psmallmatrix}=-1$ is both its label and the entry in the middle of the 2nd nontrivial row of the frieze.

\vspace{1.3ex}

\noindent \centering
\begin{tikzpicture}[baseline=(current bounding box.center)]
	\begin{scope}[scale=1.2,xshift=0cm]
	\newdimen\R
	\R=1cm
	\pgfmathsetmacro{\th}{360/8}
	
	\foreach[count=\i] \lab in {A,B,C,D,E,F,G,H} {
		\coordinate (\lab) at (90-\i*\th:\R);
		\draw[very thick, directed] (\lab) -- (90-\i*\th-\th:\R);
	}
	\node [circle,fill=black,inner sep=0pt,minimum size=1.5mm] at (H) {};
	\draw [thick, magenta] (B) -- (E);
	\draw [thick, magenta] (G) -- (E);
	\draw [thick, magenta] (G) -- (A);

	\draw [thick, dashed, green] (B) -- (D);
	\draw [thick, dashed, green] (A) -- (E);
	\node at (0.06*\R,-0.72*\R) {\footnotesize $2$};
	\node at (-0.2*\R,0.11*\R) {\footnotesize $-2$};
	\node at (0.5*\R,-0.02*\R) {\footnotesize $-1$};
	\node at (0.1*\R,0.58*\R) {\footnotesize $-1$};
	\node at (-0.5*\R,0.45*\R) {\footnotesize $-1$};
	\end{scope}
\end{tikzpicture}
\qquad
\begin{tikzpicture}[baseline=(current bounding box.center)]
	\begin{scope}[scale=1.2,xshift=0cm]
	\coordinate (A) at (0,0);
	\coordinate (X) at (1,0);
	\coordinate (B) at (0.5,1);
	\coordinate (C) at (1.5,1);
	\coordinate (D) at (2,0);
	\coordinate (E) at (1.6,1.1);
	\coordinate (F) at (1,2);
	\coordinate (G) at (0.4,1.1);
	\coordinate (H) at (-0.5,1);
	\node[above=0.1]  at (A) {\tiny $\tfrac01$};
	\node[above=0.05]  at (X) {\tiny $\tfrac12$};
	\node[right=-0.1]  at (D) {\tiny $\tfrac35$};
	\node[right]  at (E) {\tiny $\tfrac23$};
	\node[below]  at (F) {\tiny $\tfrac34$};
	\node[above]  at (G) {\tiny $\tfrac11$};
	\node[above]  at (H) {\tiny $\tfrac10$};
	\draw[very thick, directedbis] (H) -- (A);
	\draw[very thick, directedbis] (A) -- (B);
	\draw[very thick, directedbis] (B) -- (C);
	\draw[very thick, directedbis] (C) -- (D);
	\draw[very thick, directedbis] (D) -- (E);
	\draw[very thick, directedbis] (E) -- (F);
	\draw[very thick, directedbis] (F) -- (G);
	\draw[very thick, directedbis] (G) -- (H);
	\node [circle,fill=black,inner sep=0pt,minimum size=1.5mm] at (H) {};

	\draw [gray] (X) -- (A);
	\draw [gray] (X) -- (B);
	\draw [gray] (X) -- (C);
	\draw [gray] (X) -- (D);

	\draw [thick, magenta] (A) -- (G);
	\draw [thick, magenta] (E) -- (G);
	\draw [thick, magenta] (B) -- (E);
	\draw [thick, dashed, green] (B) -- (D);
	\draw [thick, dashed, green] (A) -- (E);
	\end{scope}
\end{tikzpicture}
~
{\setlength{\arraycolsep}{0.3ex} \renewcommand{\arraystretch}{0.8}
\(
\begin{matrix}
$\triv0$&&$\triv0$&&$\triv0$&&$\triv0$&&$\triv0$&&$\triv0$&&$\triv0$&&$\triv0$\\
&$\triv1$&&$\triv1$&&$\triv1$&&$\triv1$&&$\triv1$&&$\triv1$&&$\triv1$\\
        &&\clap{$-\!1$}&&$2$&&$2$&&$0$&&\clap{$-\!3$}&&\clap{$-\!1$}\\
             &&&$\clap{$-\!3$}$&&$3$&&\clap{$-\!1$}&&\clap{$-\!1$}&&$2$\\
&&&&\clap{$-\!5$}&&\clap{$-\!2$}&&$1$&&$1$\\
                       &&&&&3&&$3$&&$0$\\
                          &&&&&&\clap{$-\!4$}&&\clap{$-\!1$}\\
&&&&&&&$\triv1$\\
\end{matrix}
\)
}
\end{example}

\begin{proof}[Proof of Theorem~\ref{thm:poly:friezes}]
Cluster variables satisfy Ptolemy relations, which are just mutation relations in the division-free form. The steps of specialisation from the theorem are well-defined (by Theorem~\ref{thm:poly:main}) ring homomorphisms from the cluster algebra of type $A_n$ to the ring of polynomials with integer coefficients and then to $\Z$, so the frieze entries resulting from specialisation also satisfy Ptolemy relation; diamond rule and tameness condition follow. So the two-step specialisation described in the theorem indeed produces a tame frieze. It remains to show that every tame integral frieze arises this way.

By Theorem~\ref{thm:poly:model}, each tame integral frieze arises from a labelled triangulated polygon $T$ whose labels are the corresponding frieze entries. On the first step, we produce a specialised cluster algebra by specialising the initial cluster variables labelled $\pm1$. Every remaining diagonal is non-unit and is enclosed in a polycule and therefore the neighbourhoods of the corresponding quiver vertices are polycules too, so by Theorem~\ref{thm:poly:main}, the corresponding initial cluster variables are polyamorous. Therefore their specialisation to the labels of the diagonals is defined and produces an integral frieze. Since by Theorem~\ref{thm:poly:model} the entries of this frieze corresponding to the diagonals of $T$ are equal to their labels, this is the same frieze from which we started, as desired.
\end{proof}

\begin{example}
When the cluster algebra of type $A_2$ is specialised to $x_2=-1$ (as in Example~\ref{ex:spec}), the fundamental domain shown in Example~\ref{ex:genfr} gives rise to the family of friezes of the following form:
\[
\begin{matrix}
&\triv0 && \triv0 && \triv0 && \triv0 && \triv0 && \triv0 && \triv0 &\\
&& \triv1 && \triv1 && \triv1 && \triv1 && \triv1 && \triv1 \\
&\clap{$-1$} && \fundom{$x_1$} && \fundom{0} && \fundom{$-1\!-\!x_1$} && \clap{$-1$} && \clap{$-1$} && 0\\
\cdots&& \clap{$-1\!-\!x_1$} && \fundom{$-1$} && \fundom{$-1$}                 && x_1 && 0        && \clap{$-1\!-\!x_1$} &&\cdots\\
&\triv1 && \triv1 && \triv1 && \triv1 && \triv1 && \triv1 && \triv1\\
&& \triv0 && \triv0 && \triv0 && \triv0 && \triv0 && \triv0 &\\ 
\end{matrix}
\]
Any integral frieze with 2 nontrivial rows is either positive (i.e., it is a Conway--Coxeter frieze) or of this form (with $x_1$ specialised to an integer) or related to this form by a translation or a~glide reflection.
\end{example}

\begin{remark}
\label{remark:poly}
\begin{enumerate}[label=(\roman*)]
\item \label{remark:poly:CuntzHolm} The model for tame integral friezes given by Theorem~\ref{thm:poly:model} is closely related to the model given by \cite[Theorem 7.3]{CuntzHolmFriezes} (which therefore could be used for an alternative proof of Theorem~\ref{thm:poly:model}\ref{it:surj}). Specifically, there is a bijection from the set of labelled triangulated polygons from Definition~\ref{def:labtr} and Theorem~\ref{thm:poly:model} to the set of admissible labellings from \cite[Definition~7.1 and Theorem~7.3]{CuntzHolmFriezes} given by labelling each triangle by the product of the labels of its sides.
\item Our model is also closely related to the combinatorial model given in \cite[Section 1 and Theorem 3.2]{CuntzModel}, which, too, exploits partition of a polygon into triangles and quadrilaterals, although the data slightly differs. Our proof follows similar underlying logic.
\item Like the models from \cite{CuntzHolmFriezes} and \cite{CuntzModel}, our model is not bijective: different labelled triangulated polygons can correspond to the same integral frieze.
\item Our construction does not allow to use a completely arbitrary seed to obtain a given frieze by specialisation, but allows including an initial cluster variable $x_j$ that gets specialised to zero if a single mutation at $j$ yields a cluster variable $x_j'$ that also needs to be specialised to~0.
\item The exceptional case of \cite[Theorem 5.4 and Example 5.5]{CuntzHolmFriezes} -- a frieze whose first nontrivial row consists entirely of 0s -- corresponds to a labelled triangulated polygon that consists entirely of polycules with every non-unit diagonal labelled by~0.
\end{enumerate}
\end{remark}

\section*{Acknowledgements}
This research is supported by EPSRC grant EP/W524098/1.
I am grateful to Antoine de Saint Germain, Matty van Son, Ian Short, Anna Felikson and Pavel Tumarkin for useful discussions.
Part of this work has been done during my visit to the Department of Mathematical Sciences of Durham University.

\printbibliography

\end{document}